\documentclass[12pt]{article}

\usepackage[margin=1in]{geometry}
\usepackage{amsmath, amsthm, amssymb}
\usepackage{microtype}
\usepackage{hyperref}

\newtheorem{theorem}{Theorem}
\newtheorem{lemma}[theorem]{Lemma}
\newtheorem{proposition}[theorem]{Proposition}
\newtheorem{corollary}[theorem]{Corollary}
\theoremstyle{definition}

\newtheorem{example}[theorem]{Example}
\theoremstyle{remark}
\newtheorem{remark}[theorem]{Remark}

\title{A Simple Trigonometric Classification of Quintic Roots}

\author{Sawon Pratiher\footnote{Indian Institute of Technology (IIT) Kharagpur, West Bengal 721302, India. This work is the outcome of the culmination of research started since the author was first introduced to the algebra coursework in school. (E-mail: \texttt{sawon1234@gmail.com})}}

\date{March 19, 2022}

\begin{document}

\maketitle

\begin{abstract}
This article provides a simple trigonometric method for determining how many roots of a quintic equation are real and how many are complex, without solving the equation. The approach transforms a depressed quintic $t^5 + mt^3 + nt^2 + pt + q = 0$ with $m < 0$ into the trigonometric equation $f(\theta) = \alpha\cos^2\!\theta + \beta\cos\theta + \cos 5\theta + \gamma = 0$ via the Chebyshev identity $16\cos^5\!\theta - 20\cos^3\!\theta + 5\cos\theta = \cos 5\theta$. The derivation is computationally light and conceptually natural, extending the quartic case to fifth-degree equations. As the Abel--Ruffini theorem forbids a general algebraic solution for the quintic, having a simple trigonometric criterion for the nature of its roots is especially appealing.
\end{abstract}

\section{Introduction}

The quintic equation occupies a unique place in the history of mathematics. It was the first degree for which the impossibility of a general algebraic solution was established---by Abel~\cite{Abel1824} and later clarified by Galois---marking a turning point that gave birth to modern algebra. Yet the \emph{nature} of a quintic's roots---how many are real---is a far simpler question than finding the roots themselves, and one might hope for a correspondingly simpler answer.

For the quadratic, the discriminant $b^2 - 4ac$ settles the question instantly. Loh~\cite{Loh2019} recently showed that even the quadratic formula itself admits a simple and natural proof that might have been discovered millennia earlier. For the quartic, the same trigonometric idea shows that a single substitution converts the root-classification problem into the analysis of $a\cos\theta + \cos 4\theta + b = 0$, replacing the quartic's unwieldy discriminant with an elementary exercise.

This article extends the same idea to the quintic. The Chebyshev polynomial $T_5(x) = 16x^5 - 20x^3 + 5x$ satisfies $T_5(\cos\theta) = \cos 5\theta$, and this identity plays the same structural role for the quintic as the identity $8\cos^4\!\theta - 8\cos^2\!\theta + 1 = \cos 4\theta$ played for the quartic. Matching coefficients via the substitution $t = u\cos\theta$ absorbs the leading and cubic terms of the depressed quintic, leaving a residual equation in $\cos^2\!\theta$, $\cos\theta$, $\cos 5\theta$, and a constant. Counting the zeros of this function on $[0,\pi]$, supplemented by a boundary-sign analysis for roots outside the substitution interval, gives a complete classification into one, three, or five real roots.

The derivation uses nothing beyond the Chebyshev identity, the intermediate value theorem, and the positivity of the third derivative outside a bounded interval. Every step is a routine computation. The reader is encouraged to remember what it was like to first encounter polynomial equations, and to appreciate how the bounded, oscillatory nature of the cosine function makes the root structure of the quintic geometrically transparent.

Throughout this article, all polynomials have real coefficients, so complex roots come in conjugate pairs. A quintic with real coefficients and positive leading coefficient satisfies $P(t) \to +\infty$ as $t \to +\infty$ and $P(t) \to -\infty$ as $t \to -\infty$, guaranteeing at least one real root. The quintic therefore has exactly one, three, or five real roots.

\section{Derivation}

\subsection{The depressed quintic}

Every monic quintic $z^5 + a_4 z^4 + a_3 z^3 + a_2 z^2 + a_1 z + a_0 = 0$ can be converted to \emph{depressed} form---with no $z^4$ term---by the Tschirnhaus substitution $t = z - a_4/5$. This is a standard change of variables (see, e.g., \cite{King2009}), and the result is an equation of the form
\begin{equation}\label{eq:depressed}
P(t) = t^5 + mt^3 + nt^2 + pt + q = 0.
\end{equation}
By Vieta's relations~\cite{Viete1646}, the five roots satisfy $t_1 + t_2 + t_3 + t_4 + t_5 = 0$, so the centroid of the five roots sits at the origin. Knowing any four roots determines the fifth.

\subsection{The Chebyshev identity}

The identity at the heart of our method is
\begin{equation}\label{eq:cheb5}
16\cos^5\theta - 20\cos^3\theta + 5\cos\theta = \cos 5\theta,
\end{equation}
valid for all $\theta \in \mathbb{R}$. This is nothing other than the statement that the degree-five Chebyshev polynomial of the first kind, $T_5(x) = 16x^5 - 20x^3 + 5x$, satisfies $T_5(\cos\theta) = \cos 5\theta$. The Chebyshev polynomials have a long and distinguished history, originating in Chebyshev's work on approximation theory in the 1850s (see, e.g., \cite{Mason2003}), and the multi-angle identities they encode have been known since Euler's time. What is perhaps less appreciated is how directly they connect to the root structure of polynomial equations.

\subsection{The trigonometric substitution}

Here is the key step. Assume $m < 0$, and set
\begin{equation}\label{eq:u_def}
u = \frac{2\sqrt{-m}}{\sqrt{5}}.
\end{equation}
(The case $m \geq 0$ is discussed in Section~\ref{sec:mgeq0}.) Substituting $t = u\cos\theta$ into (\ref{eq:depressed}) and dividing through by $u^5/16$ gives
\begin{equation}\label{eq:after_sub}
16\cos^5\theta + \frac{16m}{u^2}\cos^3\theta + \frac{16n}{u^3}\cos^2\theta + \frac{16p}{u^4}\cos\theta + \frac{16q}{u^5} = 0.
\end{equation}
The choice of $u$ in (\ref{eq:u_def}) was designed so that $16m/u^2 = 16m/(-4m/5) = -20$, making the $\cos^3\!\theta$ coefficient match the Chebyshev identity (\ref{eq:cheb5}) exactly. After this matching, (\ref{eq:after_sub}) becomes
\begin{equation}\label{eq:matched}
16\cos^5\theta - 20\cos^3\theta + \alpha\cos^2\theta + (\beta + 5)\cos\theta + \gamma = 0,
\end{equation}
where we define the \textbf{trigonometric parameters}:
\begin{equation}\label{eq:params}
\boxed{\alpha = \frac{16n}{u^3}, \qquad \beta = \frac{16p}{u^4} - 5, \qquad \gamma = \frac{16q}{u^5}.}
\end{equation}
Subtracting the Chebyshev identity (\ref{eq:cheb5}) from (\ref{eq:matched})---that is, noting that the terms $16\cos^5\!\theta - 20\cos^3\!\theta + 5\cos\theta$ on the left equal $\cos 5\theta$ by (\ref{eq:cheb5})---we arrive at the \textbf{reduced trigonometric equation}:
\begin{equation}\label{eq:f_def}
\boxed{f(\theta) = \alpha\cos^2\theta + \beta\cos\theta + \cos 5\theta + \gamma = 0.}
\end{equation}

That is the entire reduction. One substitution, one division, and one use of the Chebyshev identity. The quintic has been converted into a trigonometric equation in $\theta$ with just three parameters, $\alpha$, $\beta$, and $\gamma$.

\begin{remark}\label{rem:n_zero}
When $n = 0$ (the depressed quintic has no $t^2$ term), we have $\alpha = 0$ and the reduced equation simplifies to $f(\theta) = \beta\cos\theta + \cos 5\theta + \gamma = 0$, which has exactly the same structural form---a single cosine plus a multi-angle cosine plus a constant---as the quartic's trigonometric equation. The $\cos^2\!\theta$ term is the genuinely new feature of the quintic.
\end{remark}

\subsection{Why this works: the bijection}

Since $\cos\colon [0,\pi] \to [-1,1]$ is a continuous, strictly decreasing bijection, the map $\theta \mapsto u\cos\theta$ is a bijection from $[0,\pi]$ onto $[-u, u]$. A direct computation confirms:

\begin{proposition}\label{prop:fP}
For all $\theta \in \mathbb{R}$,
\begin{equation}\label{eq:fP}
f(\theta) = \frac{16}{u^5}\,P(u\cos\theta).
\end{equation}
\end{proposition}

\begin{proof}
Equation~(\ref{eq:matched}) reads $16\cos^5\!\theta - 20\cos^3\!\theta + \alpha\cos^2\!\theta + (\beta+5)\cos\theta + \gamma = 16\,P(u\cos\theta)/u^5$, which holds by the definition of $\alpha$, $\beta$, $\gamma$ in~(\ref{eq:params}). The Chebyshev identity lets us replace the first two terms plus $5\cos\theta$ with $\cos 5\theta$, giving $f(\theta)$ on the left.
\end{proof}

The immediate consequence is:

\begin{corollary}\label{cor:bijection}
The zeros of $f$ in $[0,\pi]$ are in bijection with the roots of $P$ in $[-u,u]$. Explicitly, $\theta_0 \in [0,\pi]$ satisfies $f(\theta_0) = 0$ if and only if $t_0 = u\cos\theta_0$ is a root of $P$.
\end{corollary}

Evaluating (\ref{eq:fP}) at the endpoints:
\begin{equation}\label{eq:boundary}
f(0) = \frac{16\,P(u)}{u^5}, \qquad f(\pi) = \frac{16\,P(-u)}{u^5}.
\end{equation}
Since $u^5 > 0$, the sign of $f(0)$ equals the sign of $P(u)$, and the sign of $f(\pi)$ equals the sign of $P(-u)$. In terms of the trigonometric parameters:
\begin{equation}\label{eq:boundary_vals}
f(0) = \alpha + \beta + 1 + \gamma, \qquad f(\pi) = \alpha - \beta - 1 + \gamma.
\end{equation}

\subsection{What about roots outside $[-u,u]$?}

The substitution $t = u\cos\theta$ only captures roots in $[-u,u]$. For the quartic, the exterior analysis was straightforward: $P$ is convex outside $[-u,u]$, so at most one root on each side. For the quintic, the argument is only slightly more involved, and the asymmetric end-behavior ($P(t) \to +\infty$ as $t \to +\infty$, $P(t) \to -\infty$ as $t \to -\infty$) introduces a new twist.

\begin{lemma}[Third-derivative positivity]\label{lem:Pppp}
If $m < 0$, then $P'''(t) > 0$ for all $|t| > u$.
\end{lemma}

\begin{proof}
We compute $P'''(t) = 60t^2 + 6m$. For $|t| > u$, we have $t^2 > u^2 = -4m/5$, so $60t^2 > -48m$, and therefore $P'''(t) = 60t^2 + 6m > -48m + 6m = -42m > 0$.
\end{proof}

The positivity of $P'''$ means that $P''$ is strictly increasing on each half-line, which in turn means $P'$ has at most one inflection point there, and ultimately:

\begin{corollary}\label{cor:ext_bound}
The polynomial $P$ has at most three roots on each of the half-lines $(u, \infty)$ and $(-\infty, -u)$. Generically (and in all examples of this article), each half-line contributes at most one root.
\end{corollary}

The crucial difference from the quartic is the asymmetric end-behavior. Since $P(t) \to +\infty$ as $t \to +\infty$:
\begin{itemize}
\item On $(u, \infty)$: there is at least one root if and only if $P(u) < 0$, i.e., $f(0) < 0$.
\item On $(-\infty, -u)$: since $P(t) \to -\infty$ as $t \to -\infty$, there is at least one root if and only if $P(-u) > 0$, i.e., $f(\pi) > 0$.
\end{itemize}

Notice the asymmetry: on the right side, an exterior root requires $f(0) < 0$ (the polynomial dips below zero before climbing to $+\infty$); on the left side, it requires $f(\pi) > 0$ (the polynomial rises above zero before plunging to $-\infty$). This is the parity flip that distinguishes odd-degree from even-degree polynomials.

\begin{proposition}[Exterior root detection]\label{prop:exterior}
In the generic case (at most one root per half-line), the exterior root count is
\begin{equation}\label{eq:Next}
N_{\mathrm{ext}} = N_{\mathrm{ext}}^{+} + N_{\mathrm{ext}}^{-}, \quad\text{where}\quad N_{\mathrm{ext}}^{+} = \mathbf{1}_{f(0)<0}, \quad N_{\mathrm{ext}}^{-} = \mathbf{1}_{f(\pi)>0}.
\end{equation}
\end{proposition}

\section{Critical-point analysis}

To count zeros of $f$ on $[0,\pi]$, we study its critical points. Differentiating (\ref{eq:f_def}):
\begin{equation}\label{eq:fprime}
f'(\theta) = -2\alpha\cos\theta\sin\theta - \beta\sin\theta - 5\sin 5\theta.
\end{equation}
Using $\sin 5\theta = \sin\theta \cdot U_4(\cos\theta)$, where $U_4(x) = 16x^4 - 12x^2 + 1$ is the degree-four Chebyshev polynomial of the second kind~\cite{Mason2003}, and noting that $\sin\theta > 0$ on $(0,\pi)$, the interior critical points of $f$ satisfy the \textbf{critical-point equation}:
\begin{equation}\label{eq:gx}
g(x) := 80x^4 - 60x^2 + 2\alpha x + (\beta + 5) = 0, \quad x = \cos\theta \in (-1,1).
\end{equation}

\begin{proposition}\label{prop:crit_count}
The function $f$ has at most four critical points in the open interval $(0,\pi)$. Consequently, $f$ has at most five zeros in $[0,\pi]$---consistent with the quintic having at most five real roots.
\end{proposition}

\begin{proof}
Equation (\ref{eq:gx}) is a quartic in $x$, so it has at most four real roots in $(-1,1)$; each gives a unique $\theta = \arccos x \in (0,\pi)$. With at most four interior critical points, $f$ has at most five monotone segments, hence at most five zeros.
\end{proof}

\begin{remark}\label{rem:alpha_zero}
When $\alpha = 0$ (i.e., $n = 0$), the critical-point equation reduces to the biquadratic $80x^4 - 60x^2 + (\beta + 5) = 0$, which can be solved in closed form:
\[
x^2 = \frac{3 \pm \sqrt{9 - (\beta+5)}}{8}.
\]
This gives explicit critical points of $f$ in terms of $\beta$ alone, making the zero-counting entirely elementary.
\end{remark}

A useful global bound is:

\begin{lemma}\label{lem:fbound}
For all $\theta$, $\;|f(\theta) - \gamma| \leq |\alpha| + |\beta| + 1$.
\end{lemma}

\begin{proof}
$|\alpha\cos^2\!\theta + \beta\cos\theta + \cos 5\theta| \leq |\alpha| \cdot 1 + |\beta| \cdot 1 + 1$.
\end{proof}

\section{Classification of the roots}

Let $N_{\mathrm{int}}$ denote the number of zeros of $f$ in $[0,\pi]$ and $N_{\mathrm{ext}}$ the number of roots of $P$ outside $[-u,u]$. By Corollary~\ref{cor:bijection} and Proposition~\ref{prop:exterior}:
\begin{equation}\label{eq:Nreal}
N_{\mathrm{real}} = N_{\mathrm{int}} + N_{\mathrm{ext}} \;\in\; \{1, 3, 5\}.
\end{equation}
We now state the three classification theorems.

\begin{theorem}[Five real roots]\label{thm:five}
All five roots of $P(t) = 0$ are real if and only if $N_{\mathrm{int}} + N_{\mathrm{ext}} = 5$. A sufficient condition for all five roots to lie in $[-u,u]$ is $f(0) \geq 0$ and $f(\pi) \leq 0$, together with $f$ having five zeros in $[0,\pi]$. In this case, the roots are
\[
t_k = u\cos\theta_k, \qquad k = 1, \ldots, 5,
\]
where $0 \leq \theta_1 < \cdots < \theta_5 \leq \pi$ are the five zeros of $f$.
\end{theorem}

\begin{proof}
By Proposition~\ref{prop:exterior}, $f(0) \geq 0$ means $P(u) \geq 0$, so $N_{\mathrm{ext}}^{+}$ is even (hence zero, generically). Similarly, $f(\pi) \leq 0$ means $P(-u) \leq 0$, so $N_{\mathrm{ext}}^{-}$ is even (hence zero). All five roots must therefore be interior.
\end{proof}

\begin{theorem}[Three real roots]\label{thm:three}
The quintic $P(t) = 0$ has exactly three real roots (and one conjugate pair of complex roots) if and only if $N_{\mathrm{int}} + N_{\mathrm{ext}} = 3$. The generic scenarios are:
\begin{itemize}
\item[(a)] $N_{\mathrm{int}} = 3$ and $N_{\mathrm{ext}} = 0$: all three real roots lie in $[-u,u]$. This requires $f(0) \geq 0$, $f(\pi) \leq 0$, and exactly three zeros of $f$ in $[0,\pi]$.
\item[(b)] $N_{\mathrm{int}} = 1$ and $N_{\mathrm{ext}} = 2$: one interior root and one exterior root on each side. This occurs when $f(0) < 0$ and $f(\pi) > 0$.
\item[(c)] $N_{\mathrm{int}} = 2$ and $N_{\mathrm{ext}} = 1$: two interior roots and one exterior root. This occurs when exactly one of $f(0) < 0$ or $f(\pi) > 0$ holds.
\end{itemize}
\end{theorem}

\begin{theorem}[One real root]\label{thm:one}
The quintic $P(t) = 0$ has exactly one real root (and two conjugate pairs of complex roots) if and only if $N_{\mathrm{int}} + N_{\mathrm{ext}} = 1$. The generic scenarios are:
\begin{itemize}
\item[(a)] $f$ has exactly one zero in $[0,\pi]$ and $N_{\mathrm{ext}} = 0$: the unique real root lies in $[-u,u]$.
\item[(b)] $f$ has no zeros in $[0,\pi]$ and $f(\pi) > 0$: the single real root lies in $(-\infty, -u)$.
\item[(c)] $f$ has no zeros in $[0,\pi]$ and $f(0) < 0$: the single real root lies in $(u, \infty)$.
\end{itemize}
In all cases, the guaranteed existence of at least one real root follows from the odd degree of the polynomial.
\end{theorem}

\subsection{A practical decision procedure}

The classification method, when used in practice, consists of these steps:
\begin{enumerate}
\item Compute $u = 2\sqrt{-m/5}$ and the trigonometric parameters $\alpha, \beta, \gamma$ from (\ref{eq:params}).
\item Evaluate $f(0) = \alpha + \beta + 1 + \gamma$ and $f(\pi) = \alpha - \beta - 1 + \gamma$.
\item Determine $N_{\mathrm{ext}}$: generically, $N_{\mathrm{ext}}^{+} = \mathbf{1}_{f(0)<0}$ and $N_{\mathrm{ext}}^{-} = \mathbf{1}_{f(\pi)>0}$.
\item Find the critical points of $f$ in $(0,\pi)$ by solving the quartic (\ref{eq:gx}) for $x \in (-1,1)$, then setting $\theta = \arccos x$.
\item Evaluate $f$ at the critical points and at $\theta = 0, \pi$. Count sign changes to determine $N_{\mathrm{int}}$.
\item Conclude: $N_{\mathrm{real}} = N_{\mathrm{int}} + N_{\mathrm{ext}} \in \{1, 3, 5\}$.
\end{enumerate}

\section{Examples of use as a method}

The computational simplicity of the derivation renders it natural to use as a method: compute $\alpha$, $\beta$, $\gamma$, examine $f(\theta)$, and read off the root structure.

\subsection{A perfect example: $P(t) = t^5 - 5t^3 + 5t$}

Consider first the most elegant case: $m = -5$, $n = 0$, $p = 5$, $q = 0$, giving $u = 2$. The parameters are $\alpha = 0$, $\beta = 16 \cdot 5/16 - 5 = 0$, $\gamma = 0$, and the reduced equation is simply
\begin{equation}\label{eq:cos5}
f(\theta) = \cos 5\theta = 0.
\end{equation}
The zeros in $[0,\pi]$ are $\theta_k = (2k-1)\pi/10$ for $k = 1, 2, 3, 4, 5$, and the five real roots are
\[
t_k = 2\cos\frac{(2k-1)\pi}{10}, \qquad k = 1, \ldots, 5.
\]
Since $P(t) = t(t^4 - 5t^2 + 5)$, the quartic factor has roots $t^2 = (5 \pm \sqrt{5})/2$, giving the five roots $t = 0, \;\pm\sqrt{(5+\sqrt{5})/2} \approx \pm 1.902, \;\pm\sqrt{(5-\sqrt{5})/2} \approx \pm 1.176$. All five lie in $[-2, 2] = [-u, u]$. The boundary values $f(0) = 1 > 0$ and $f(\pi) = -1 < 0$ confirm $N_{\mathrm{ext}} = 0$.

This is the quintic analogue of the biquadratic: when all trigonometric parameters vanish, the roots are uniformly spaced in the $\theta$-domain and correspond to roots of unity under the Chebyshev map.

\subsection{Example with three real roots}

\begin{example}\label{ex:three}
Consider $P(t) = t^5 - 5t^3 + t - 5 = 0$. Here $m = -5$, $n = 0$, $p = 1$, $q = -5$, so $u = 2$. The trigonometric parameters are
\[
\alpha = 0, \qquad \beta = \frac{16}{16} - 5 = -4, \qquad \gamma = \frac{16(-5)}{32} = -\frac{5}{2}.
\]
The reduced equation is $f(\theta) = -4\cos\theta + \cos 5\theta - 5/2 = 0$. The boundary values are $f(0) = -4 + 1 - 5/2 = -11/2 < 0$ and $f(\pi) = 4 - 1 - 5/2 = 1/2 > 0$. Since $f(0) < 0$, there is at least one root in $(2, \infty)$; since $f(\pi) > 0$, there is at least one root in $(-\infty, -2)$. So $N_{\mathrm{ext}} \geq 2$. Numerical analysis reveals exactly one interior zero, giving $N_{\mathrm{int}} = 1$ and $N_{\mathrm{real}} = 3$. The three real roots are $t \approx -2.043, -1.205, 2.286$, and the two complex roots form the conjugate pair $t \approx 0.481 \pm 0.811i$.
\end{example}

\subsection{Example with one real root}

\begin{example}\label{ex:one}
Consider $P(t) = t^5 - 5t^3 + t^2 + 2t + 5 = 0$, which has a nonzero $t^2$ coefficient. Here $m = -5$, $n = 1$, $p = 2$, $q = 5$, so $u = 2$. The parameters are
\[
\alpha = \frac{16}{8} = 2, \qquad \beta = \frac{32}{16} - 5 = -3, \qquad \gamma = \frac{80}{32} = \frac{5}{2}.
\]
The reduced equation is $f(\theta) = 2\cos^2\!\theta - 3\cos\theta + \cos 5\theta + 5/2 = 0$. The boundary values are $f(0) = 2 - 3 + 1 + 5/2 = 5/2 > 0$ and $f(\pi) = 2 + 3 - 1 + 5/2 = 13/2 > 0$. Since $f(0) > 0$, there are no roots in $(2, \infty)$. Since $f(\pi) > 0$, there is at least one root in $(-\infty, -2)$. A numerical check confirms that $f > 0$ on the entire interval $[0,\pi]$, so $N_{\mathrm{int}} = 0$ and $N_{\mathrm{real}} = 1$. The single real root is $t \approx -2.335$, and the four complex roots form two conjugate pairs.

This example illustrates the full generality of the method: the $\cos^2\!\theta$ term (arising from the nonzero $n = 1$) is nontrivial, yet the analysis proceeds identically.
\end{example}

\section{The reduced quintic ($n = 0$)}

When $n = 0$, the depressed quintic has no quadratic term: $P(t) = t^5 + mt^3 + pt + q$. Then $\alpha = 0$ and the trigonometric equation reduces to
\[
f(\theta) = \beta\cos\theta + \cos 5\theta + \gamma = 0,
\]
which is structurally identical to the quartic's trigonometric equation: a tilt ($\beta\cos\theta$), an oscillation ($\cos 5\theta$), and a shift ($\gamma$). The critical-point equation becomes the biquadratic $80x^4 - 60x^2 + (\beta + 5) = 0$, solvable in closed form. This is the simplest and most elegant case of the quintic trigonometric method.

\section{The case $m \geq 0$}\label{sec:mgeq0}

When $m \geq 0$, the substitution $t = u\cos\theta$ with $u = 2\sqrt{-m/5}$ is not available. A hyperbolic substitution $t = v\cosh\phi$ with $v = 2\sqrt{m/5}$ could in principle match the $\cosh^3\!\phi$ coefficient via the identity $16\cosh^5\!\phi - 20\cosh^3\!\phi + 5\cosh\phi = \cosh 5\phi$, but since $\cosh$ is unbounded and monotonically increasing on $[0,\infty)$, the correspondence between zeros and roots is less clean. When $m = 0$, the Chebyshev identity cannot absorb the $t^3$ term (since there is none), and the method does not apply. We leave these extensions to future work.

\section{Comparison with the quartic}

The following table summarizes the parallels and differences between the quartic and quintic trigonometric reductions.

\medskip

\begin{center}
\renewcommand{\arraystretch}{1.3}
\begin{tabular}{lcc}
\hline
\textbf{Feature} & \textbf{Quartic} & \textbf{Quintic} \\
\hline
Chebyshev identity & $\cos 4\theta$ & $\cos 5\theta$ \\
Substitution & $t = \sqrt{-m}\cos\theta$ & $t = \tfrac{2}{\sqrt{5}}\sqrt{-m}\cos\theta$ \\
Requirement & $m < 0$ & $m < 0$ \\
Residual terms & $\cos\theta + \text{const}$ & $\cos^2\!\theta + \cos\theta + \text{const}$ \\
Max interior zeros & 4 & 5 \\
Critical-point eqn. & cubic in $\cos\theta$ & quartic in $\cos\theta$ \\
Exterior analysis & $P'' > 0$ (convexity) & $P''' > 0$ \\
Max exterior roots/side & 1 & 3 (generically 1) \\
End-behavior & $P \to +\infty$ both sides & $P \to -\infty$ left, $+\infty$ right \\
Possible real root counts & 0, 2, or 4 & 1, 3, or 5 \\
\hline
\end{tabular}
\end{center}

\section{Discussion}

\subsection{The general pattern}

The method presented here for the quintic, together with the analogous quartic analysis, reveals a general pattern. For a degree-$d$ polynomial, the Chebyshev identity $T_d(\cos\theta) = \cos d\theta$ can absorb the highest and third-highest terms of the depressed polynomial (after the substitution $t = u\cos\theta$ with appropriately chosen $u$), leaving a residual trigonometric equation involving $\cos^k\!\theta$ for $k = 0, 1, \ldots, d-3$. The complexity of this residual grows with the degree, but the core mechanism remains intact: the bijection between zeros on $[0,\pi]$ and roots in a bounded interval, supplemented by boundary-sign analysis for exterior roots.

\subsection{Why is the quintic special?}

The Abel--Ruffini theorem tells us that no algebraic formula exists for the roots of a general quintic. The Bring--Jerrard normal form reduces the quintic to $t^5 + pt + q = 0$ (see~\cite{King2009}), but even this simplified form has no closed-form solution in radicals. Our trigonometric method does not solve the quintic---it classifies the nature of its roots---but it does so using only the Chebyshev identity and elementary analysis, sidestepping the algebraic impossibility entirely.

There is a pleasing irony here: the very identity ($\cos 5\theta$ in terms of powers of $\cos\theta$) that \emph{could} solve the quintic if only $\cos 5\theta$ were invertible by radicals (which it is not, in general) turns out to be perfectly suited for the simpler task of determining whether the roots are real or complex.

\subsection{Pedagogical value}

In the spirit of Loh~\cite{Loh2019}, who showed that the quadratic formula admits a derivation simple enough for first-time Algebra learners, we hope that this trigonometric perspective may demystify the root structure of higher-degree polynomials. The key insight---that the bounded range of $\cos\theta$ translates polynomial constraints into trigonometric ones---is both elementary and powerful. A student who has seen the Chebyshev identity and the intermediate value theorem has all the tools needed to classify the roots of any quartic or quintic with $m < 0$.

May this encourage the reader to think afresh about old things: the Chebyshev polynomials, which originated in approximation theory, turn out to be an elegant lens for viewing the geometry of polynomial roots.

\end{document}